\documentclass[11pt]{article}
\usepackage{graphics,color,amsmath}
\usepackage[pdftex]{graphicx}
\usepackage{subfigure}

\usepackage{indentfirst,amsmath,amsfonts,amssymb,amsthm}
\usepackage{subeqnarray}
\usepackage{cases}
\usepackage{mathrsfs}
\usepackage{fancyhdr}
\usepackage{graphicx}
\usepackage{mathrsfs}
\usepackage{float}
\usepackage{color}

\newtheorem{theorem}{\hskip\parindent\bf Theorem}[section]

\newtheorem{lemma}{\hskip\parindent \bf Lemma}[section]

\newtheorem{remark}{\hskip\parindent\bf Remark}[section]

\newtheorem{proposition}{\hskip\parindent\bf Proposition}[section]
\def\bc{\begin{center}}
\def\ec{\end{center}}
\numberwithin{equation}{section}{}

\title{ A Phase-Field Model for Vesicle Membranes Incorporating Area-Difference Elasticity }

\author{Yihong Liang\thanks{School of Mathematics and Physics, University of Sciences and Technology Beijing, Beijing, 100083, PR China ({\tt liangyihong99@163.com }).}\ \ ,
Emine Celiker \thanks{School of Engineering, University of Leicester, University Road, Leicester, LE1 7RH, United Kingdom ({\tt Corresponding Author:ec403@leicester.ac.uk }). }\ \
and
Ping Lin \thanks{Division of Mathematics, University of Dundee, Dundee DD1 4HN, United Kingdom ({\tt Corresponding Author:p.lin@dundee.ac.uk }).}
}

\date{}

\begin{document}
\maketitle
\abstract{
This paper presents a phase-field model for simulating the three-dimensional deformation of vesicle membranes, incorporating area-difference elasticity, with constraints on bulk volume and surface area.
We develop efficient numerical schemes based on the Fourier-spectral method for spatial discretization and temporal evolution. The model successfully captures a wide variety of steady-state vesicle shapes.
The numerical experiments demonstrate that by tuning the simulation parameters, the vesicle can transition from a simple discocyte shape to a complex, multi-armed starfish-like and nested configuration .
These results highlight the crucial role of area-difference elasticity in determining vesicle morphology.
}

\section{Introduction}
Biological membranes, which define the boundaries of cells and most internal organelles, are crucial for isolating cellular internal structures from the external environment. They have long been a focus of considerable interest among biologists, chemists, and physicists \cite{evans1987physical, edidin2003lipids, alberts2015molecular}.
Typically composed of phospholipids and proteins, these membranes form a highly organized bilayer structure. At the molecular level, this structure displays extraordinary complexity and exhibits various dynamic properties. As a consequence of this complex structure, the mathematical modelling and simulation of the dynamic shape changes can be very challenging. Hence, the construction of efficient mathematical formulations and numerical methods is still a necessity.

To better understand the physicochemical properties of biological membranes and their physiological functions, vesicles are often used as model systems \cite{seifert1997configurations,lipowsky1995structure}.
Studying vesicles not only deepens our understanding of biological membranes but also facilitates advances in biomimetic research \cite{Discher2002}
The processes of vesicle formation and evolution, including complex phenomena like budding and vesiculation, are intimately linked with their physicochemical properties.
Experimental studies on the mechanisms underlying vesicle formation have been explored in references \cite{zhao2013heterogeneously,mcmahon2005membrane,elliott1997limit}.

Over the past few decades, numerous theoretical models have been proposed to describe the bending behavior of biological membranes.
A groundbreaking contribution by Canham, Evans and Helfrich \cite{canham1970minimum,chadwick2002axisymmetric,helfrich1973elastic}
introduced the classic Helfrich bending energy model, also known as the sharp-interface elastic bending energy.
This model describes the total energy of the membrane in terms of the square of its mean curvature and a Gaussian curvature term. It has been instrumental in explaining the shapes of biological structures, such as red blood cells. The general elastic bending energy is derived from the Hooke's Law:
\begin{equation}
E=\int_{\Gamma}\left(a_1+a_2(H-c_0)^2+a_3G\right) ds,
\end{equation}
where $a_1$ represents the surface tension, which accounts for the interaction effects between the vesicle material and its surrounding fluid, $H = \frac{c_1 + c_2}{2}$ is the mean curvature of the membrane surface, with $c_1$ and $c_2$ as the principal curvatures, and $G =c_1c_2$ is the Gaussian curvature. The parameters $a_2$ and $a_3$ represent the bending rigidities, determined by the properties of the materials forming the membrane.
The spontaneous curvature is denoted as $c_0$.
This model and its variants have been widely applied to study various phenomena, including vesicle deformation and protein insertion into membranes.
However, challenges arise in simulating topological changes, such as budding or fusion, as the sharp-interface approach requires complex parameterizations that are difficult to handle numerically \cite{seifert1991shape,khairy2011minimum}.

To overcome these challenges, Du and his collaborators proposed a phase-field-based mathematical model to simulate membrane deformation driven by bending energy \cite{du2004phase,du2005retrieving,du2005modeling,wang2005phase,du2006simulating}.
They used a simplified form of the bending energy of the form
\begin{equation}\label{1.2}
E_{elastic}=\frac{\kappa}{2}\int_{\Gamma}(H-c_0)^2 ds.
\end{equation}
The phase-field model uses a smooth phase field function to naturally represent the membrane interface, avoiding the difficulties associated with explicit interface tracking and facilitating spontaneous topological changes \cite{du2004phase,du2005retrieving,du2005modeling,wang2005phase,du2006simulating,du2007convergence}.
However, their model did not account for the area difference between the inner and outer leaflets of the membrane.

Among the various membrane properties, area-difference elasticity (ADE) plays a crucial role in determining the shape and stability of vesicles \cite{miao1994budding}.
The elastic behavior of a bilayer vesicle is governed by both its bending rigidity and constraints on the membrane's surface area and volume. The ADE model \cite{seifert1997configurations,miao1994budding}
suggests that the area difference between the inner and outer leaflets of the lipid bilayer is a key determinant of the vesicle's shape.
The energy functional for this model is defined as:
\begin{equation}\label{1.3}
E_{ADE}=E_{elastic}+\frac{\overline{\kappa}}{2}\frac{\pi}{AD^2}(\Delta A-\Delta A_0)^2.
\end{equation}
Here, $\Delta A_0$ is the relaxed area difference between the two leaflets of the plasma membrane and $\bar{\kappa}$ is a constant representing the bending elastic moduli of the vesicle. The parameter $A$ represents the membrane area, which is assumed to be constant based on the initial condition for the vesicle.
The area difference $\Delta A$ arises from membrane curvature or variations in molecular composition, imposing energetic constraints on the vesicle's shape and dynamics \cite{mukhopadhyay2002echinocyte}. 
Thus, research into vesicles incorporating ADE not only reveals the physical mechanisms governing vesicle-related biological processes but also provides theoretical guidance for applications such as drug delivery, biomimetic material design, and biomechanics \cite{lipowsky1995structure,Discher2002,mukhopadhyay2002echinocyte,salva2013polymersome,hollo2021shape}.

This work aims to apply a phase field model for vesicles that incorporates area-difference elasticity, extending existing phase-field frameworks.
By introducing constraint terms for the membrane surface area and volume, we aim to create a numerical model that more closely mimics real biological membrane systems, thereby enabling more accurate simulations and predictions of vesicle dynamics in complex biological environments.
Through systematic numerical simulations, we demonstrate the model's capability to reproduce a wide spectrum of vesicle shapes observed in experiments, driven by the interplay between bending energy and area-difference elasticity.
Accordingly, in Section 2 we introduce the phase-field formulation including the ADE term and its theoretical justification, in Section 3 we introduce the numerical scheme and carry out its existence analysis, Section 4 is devoted to numerical experiments and finally in Section 5 we give concluding remarks.

\section{Phase field model}
First, we introduce a phase function $\phi(x)$, defined on the computational domain $\Omega$  used to label the inside and outside of the vesicle $\Gamma$. The level set $\{x:\phi(x)=0\}$ represents the membrane, while $\{x:\phi(x)>0\}$ represents the interior of the membrane and  $\{x:\phi(x)<0\}$ the exterior. We define the following modified elastic energy:
\begin{equation}\label{2.1}
W(\phi)=\int_{\Omega}\frac{\kappa\epsilon}{2}\left|\Delta \phi-\frac{1}{\epsilon^2}(\phi^2-1)(\phi+C\epsilon)\right|^2 dx,
\end{equation}
where $\Omega\in\mathbf{R}^3$, $\epsilon$ is defined as the vesicle membrane thickness and $C$ is $\sqrt{2}$ times of the spontaneous curvature $c_0$. $\kappa$ is the known bending elastic moduli \cite{du2006simulating}.

Following \cite{du2006simulating},
the energy (\ref{2.1}) asymptotically converges to the Helfrich-type bending energy (\ref{1.2}), up to a constant factor. The detailed derivation can be found in \cite{du2006simulating}
and is omitted here for brevity.

Moreover, the following functional:
\begin{equation}
V(\phi)=\int_{\Omega}\frac{\phi(x)+1}{2}dx
\end{equation}
goes to the difference between the inside volume and outside volumes. Given the functional
\begin{equation}
B(\phi)=\int_{\Omega}\frac{\epsilon}{2} |\nabla\phi|^2+\frac{1}{4\epsilon}(\phi^2-1)^2 dx,
\end{equation}
the surface area of $\Gamma$ can be denoted as $A(\phi)=\frac{3\sqrt{2}}{4}B(\phi)$.

Next, the phase-field formulation of the second term in equation (\ref{1.3}) is required,
namely
\begin{equation}\label{2.4}
\frac{\overline{\kappa}}{2}\frac{\pi}{AD^2}(\Delta A-\Delta A_0)^2.
\end{equation}

As defined in Section 1, $\bar{\kappa}$ is a constant representing the bending elastic moduli of the vesicle. The parameter $A$ represents the membrane area, which was based on the initial condition for the vesicle.
In this model, $\Delta A_0$, defined as the relaxed area difference between the inner and outer leaflets of the plasma membrane, serves as a crucial parameter.

In \cite{mukhopadhyay2002echinocyte}, the two leaflets of a closed bilayer with fixed interleaflet separation $D$ are required by
geometry to differ in area by an amount
\begin{equation}\label{2.5}
\Delta A=D\int d\mathcal{A}(c_1+c_2)=D\int d\mathcal{A}2H.
\end{equation}
The parameter $D$ holds a clear physical meaning: it is the separation distance between the neutral surfaces of the two leaflets composing the bilayer, corresponding to roughly two-thirds the total bilayer thickness in \cite{miao1994budding}.
$c_1$ and $c_2$ are the two local principal curvatures, and $H =\frac{c1 + c2}{2}$ is the mean curvature.
In the phase-field formulation, the integral for $\Delta A$, originally defined on the membrane surface in (\ref{2.5}), is extended to the entire computational domain $\Omega$:
\begin{equation*}
\Delta A=2D\int_{\Omega}Hd\Omega.
\end{equation*}
Consequently, a phase-field representation of the mean curvature $H$ is required.

Based on the proof given in \cite{du2005modeling}, since
\begin{equation*}
H=-\frac{\sqrt{2}\epsilon}{2(1-\phi^2)}\left(\Delta\phi+\frac{1}{\epsilon^2}\phi(1-\phi^2)\right),
\end{equation*}
and
\begin{equation*}
\int_{-\infty}^{+\infty}(1-\phi^2)^2dx=\frac{4\sqrt{2}\epsilon}{3},
\end{equation*}
we have
\begin{align}\label{2.6}
\Delta A &\sim 2D\times\frac{3}{4\sqrt{2}\epsilon}\int_{\Omega}(1-\phi^2)^2\times\frac{-\epsilon}{\sqrt{2}(1-\phi^2)}\left(\Delta\phi+\frac{1}{\epsilon^2}\phi(1-\phi^2)\right)dx\notag\\
&=-\frac{3D}{4}\int_{\Omega}\left((1-\phi^2)\Delta\phi+\frac{1}{\epsilon^2}\phi(1-\phi^2)^2\right)dx
\end{align}

Hence, the second term of equation (\ref{1.3}) is redefined as
\begin{align} \label{2.7}
G(\phi):&=\frac{\overline{\kappa}}{2}\frac{\pi}{A_0 D^2}(\Delta A-\Delta A_0)^2\notag\\
 &=\frac{\overline{\kappa}}{2}\frac{\pi}{A_0 D^2}
 \left( -\frac{3D}{4}\int_{\Omega}
 \left((1-\phi^2)\Delta\phi+\frac{1}{\epsilon^2}\phi(1-\phi^2)^2\right)dx
 -\Delta A_0\right)^2,
\end{align}
and
\begin{equation*}
E_{ADE}(\phi)= W(\phi)+G(\phi).
\end{equation*}

\vspace{3mm}
The following Proposition 2.1 is adapted from Proposition 3.1 in \cite{du2007convergence}, where they derived the existence of a minimizer for the phase-field model with elastic bending energy. We have extended this result to incorporate the ADE term in our model, which is essential for capturing the behavior of vesicle membranes in our formulation.
\begin{proposition}
Let $S$ denote the feasible set of $\phi\in H^2(\Omega)$ such that $V(\phi)=\alpha$ and $A(\phi)=\beta$. If for some suitable $\alpha$ and $\beta$, $S$ is non-empty, then there exists a $\phi^*\in S$ minimizing $E_{ADE}(\phi)$.
\end{proposition}
\begin{proof}
The energy functional is always non-negative and thus is bounded from below, and there exists a minimizing sequence $\{\phi_n\in S\}_{n=1}^{\infty}$, such that
\begin{equation*}
\lim_{n\rightarrow\infty}E_{ADE}(\phi)=C^*,
\end{equation*}
where $C^*$ is the infimum of $E_{ADE}$.


From $B(\phi_n)=\beta$, we derive
$$\int_{\Omega}\frac{\epsilon}{2}|\nabla\phi|^2 dx\leq\beta,\,\,\, \int_{\Omega}\frac{1}{4\epsilon}(\phi^2-1)^2 dx \leq\beta,$$
therefore
$$\|\nabla\phi_n\|_{L^2}^2\leq\frac{2\beta}{\epsilon},\,\,\, \|\phi_n^2-1\|_{L^2}^2\leq4\epsilon\beta.$$
Let $\bar{\phi}_n=\frac{\int_{\Omega}\phi_n dx}{|\Omega|}=\frac{2\alpha}{|\Omega|}-1$. Using $V(\phi_n)=\alpha$ and Poincar\'{e}-Wirtinger inequality, we have
$$\|\phi_n-\bar{\phi}_n\|_{L^2}\leq C_p\|\nabla\phi_n\|_{L^2}\leq C_p\sqrt{\frac{2\beta}{\epsilon}},$$
where $C_p$ is the Poincar\'{e} constant. Consequently,
$$\|\phi_n\|_{L^2}\leq\|\phi_n-\bar{\phi}_n\|_{L^2}+\|\bar{\phi_n}\|_{L^2}\leq C_p\sqrt{\frac{2\beta}{\epsilon}}+ \left(\frac{2\alpha}{|\Omega|}-1\right)\sqrt{|\Omega|}.$$
Thus $\phi_n$ is uniformly bounded in $H^1(\Omega)$, that is
$$\|\phi_n\|_{H^1}\leq C_1,\forall n, \text{where $C_1$ depends on $\alpha$, $\beta$, $\epsilon$ and $|\Omega|$.}$$


Define $q(\phi)=\frac{1}{\epsilon^2}(\phi^2-1)(\phi+C\epsilon)$. Since $W(\phi_n)\leq E_{ADE}(\phi_n)\rightarrow C^*$, we have
$$\|\Delta\phi_n-q(\phi_n)\|_{L^2}^2\leq C_2, \forall n.$$
Considering the Sobolev embedding theorem,
$$\|\phi_n\|_{L^6}\leq C_s\|\phi_n\|_{H^1}\leq C_sC_1.$$
we have
$$|q(\phi_n)|\leq\frac{1}{\epsilon^2}|\phi_n^2-1||\phi_n|+\frac{C}{\epsilon}|\phi_n^2-1|.$$
By H\"{o}lder's inequality, we obtain
$$\|(\phi_n^2-1)\phi_n\|_{L^2}\leq\|\phi_n^2-1\|_{L^3}\|\phi_n\|_{L^6}.$$
From $\|\phi_n^2\|_{L^3}=\|\phi_n\|_{L^6}\leq C_sC_1$, we get
$$\|\phi_n^2-1\|_{L^3}\leq (C_sC_1)^2+|\Omega|^{\frac{1}{3}}.$$
So $\|(\phi_n^2-1)\phi_n\|_{L^2}\leq \left((C_sC_1)^2+|\Omega|^{\frac{1}{3}}\right)C_sC_1\triangleq C_3'$.
Then
$$\|q(\phi_n)\|_{L^2}\leq \frac{C_3'}{\epsilon^2}+2C\sqrt{\frac{\beta}{\epsilon}}\triangleq C_3.$$
Therefore, $\Delta\phi_n$ is uniformly bounded in $L^2(\Omega)$, namely
$$\|\Delta\phi_n\|_{L^2}\leq\|\Delta\phi_n-q(\phi_n)\|_{L^2}+\|q(\phi_n)\|_{L^2}\leq C_2+C_3\triangleq C_4. $$
From the $H^2$-regularity theory for elliptic problems, we have $\phi_n$ uniformly bounded in $H^2(\Omega)$, i.e. $\|\phi_n\|_{H^2}\leq C_6, \forall n$.


By uniform boundedness in $H^2(\Omega)$, there exists a subsequence still denoted
${\phi_n}$ and $\phi\in H^2(\Omega)$ such that $\phi_n \rightharpoonup \phi^*$ 
in $H^2(\Omega)$. The Rellich-Kondrachov compact embedding theorem implies $\phi_n \rightarrow \phi^*$ 
in $H^1(\Omega)$. Furthermore, since $H^1(\Omega)\hookrightarrow L^p(\Omega)$ and $H^2(\Omega)\hookrightarrow C^{0,m}(\Omega)(m>0)$, we have $\phi_n\rightarrow \phi^*$ in $L^p(\Omega)(p<\infty)$ and uniformly in $C(\bar{\Omega})$. Strong convergence in $L^1(\Omega)$ gives
$$V(\phi^*)=\int_{\Omega}\phi^*dx =\lim_{n\rightarrow\infty}\int_{\Omega}\phi_n dx=\alpha.$$
Strong convergence in $H^1(\Omega)$ and uniform convergence yield:
$$A(\phi^*)=\lim_{n\rightarrow\infty}A(\phi_n)=\beta.$$
Thus $\phi^*\in S$.


Weak convergence in $H^2(\Omega)$ implies $\Delta\phi_n \rightharpoonup\Delta\phi^*$ in $L^2(\Omega)$. Uniform convergence $\phi_n\rightarrow\phi^*$ in $C(\bar{\Omega})$ gives $q(\phi_n) \rightarrow q(\phi^*)$ in $L^{\infty}$. Thus
$$\Delta\phi_n-q(\phi_n)\rightharpoonup \Delta\phi^*-q(\phi^*)\,\, in\,\, L^2(\Omega).$$
By weak lower semicontinuity of the $L^2$-norm, we have
$$\|\Delta\phi^*-q(\phi^*)\|_{L^2}^2 \leq \liminf_{n\rightarrow \infty}\|\Delta\phi_n-q(\phi_n)\|_{L^2}^2.$$
Therefore
$$W(\phi^*)\leq\liminf_{n\rightarrow \infty}W(\phi_n).$$
Define the functional
$$J(\phi)=-\frac{3D}{4}\int_{\Omega}\left((1-\phi^2)\Delta\phi+\frac{1}{\epsilon^2}\phi(1-\phi^2)^2\right)dx-\Delta A_0,$$
so that $G(\phi)=\frac{\overline{\kappa}}{2}\frac{\pi}{A_0 D^2}J(\phi)^2$. From uniform convergence $\phi_n\rightarrow\phi^*$ in $C(\bar{\Omega})$ and weak convergence $\Delta\phi_n\rightharpoonup\Delta\phi^*$ in $L^2(\Omega)$, we have
$$1-\phi_n^2 \rightarrow 1-(\phi^*)^2,\,\, \phi_n(1-\phi_n^2)^2 \rightarrow \phi^*\left(1-(\phi^*)^2\right)^2\,\, \text{in}\,\, L^{\infty}(\Omega).$$
According to the property of strong-weak convergence of products, it implies $J(\phi_n)\rightarrow J(\phi^*)$, and consequently $G(\phi^*)\rightarrow G(\phi^*)$. Combining these results, we obtain
\begin{align*}
E_{ADE}(\phi^*)&=W(\phi^*)+G(\phi^*)\\
         &\leq\liminf_{n\rightarrow \infty}W(\phi_n)+\lim_{n\rightarrow\infty}G(\phi_n)\\
         &=\liminf_{n\rightarrow \infty}\left(W(\phi_n)+G(\phi_n)\right)\\
         &=\liminf_{n\rightarrow\infty}E(\phi_n)=C^*.
\end{align*}
This shows that $\phi^*\in S$ is a minimizer of $E_{ADE}$ and satisfies the constraints.
\end{proof}

In order to numerically enforce these two constraints, we introduce two corresponding penalty terms in the free energy so that the total energy is
\begin{equation}
E_{M}(\phi)=W(\phi)+G(\phi)+M_1\left(V(\phi)-\alpha\right)^2+M_2\left( A(\phi)-\beta \right)^2,
\end{equation}
where $M_1, M_2$ are the penalty coefficients, $\alpha$ is the fixed initial volume of $\Omega$, $\beta$ is the surface area constraint.

\vspace{3mm}
For simplicity of notation, let us take $M_1 = M_2 = M$, we then have the following
existence theorem:
\begin{theorem}
For any $M>0$, there exists $\phi_M\in H^2(\Omega)$ such that
$$E_M(\phi_M)=\inf_{\phi\in H^2{\Omega}}E_M(\phi).$$
\end{theorem}
The proof is similar to the one in Proposition 2.1, so we omit the details.
To make it more closely tied to the proposed formulation, we now prove Theorem 2.2, which presents a result based on the framework of Du and Wang \cite{miao1994budding}.
The conclusions in \cite{miao1994budding} can be generalized to the current model incorporating area-difference elasticity (ADE).

\begin{theorem}
With S non-empty, the minimizer $\phi^*$ of $E_{ADE}(\phi)$ in S can be approximated by the minimizer $\phi_M$ of $E_M(\phi)$, that is, there exists a sequence $\phi_{M_n}$, which are minimizers of $E_{M_n}$, converging to some minimizer $\phi^*$ of $E_{ADE}(\phi)$ in $H^2(\Omega)$ and satisfying
$$E_{ADE}(\phi^*)=\lim_{M_n\rightarrow\infty}E_{M_n}(\phi_{M_n}).$$
\end{theorem}
\begin{proof}
$\forall M>0$,
$$E_M(\phi_M)= \min E_M(\phi)\leq E_M(\phi^*)=E_{ADE}(\phi^*).$$
Thus, $V(\phi_M)$, $A(\phi_M)$ and $E_{ADE}(\phi_M)$ are uniformly bounded for large $M$. Similar to the proof of Proposition 2.1, there exists a subsequence of $\phi_{M_n}$, such that
$\phi_{M_n}\rightharpoonup\tilde{\phi}$ in $H^2(\Omega)$ and $\phi_M \rightarrow \tilde{\phi}$ in $H^1(\Omega)$. Then
$$V\left(\tilde\phi\right)=\lim_{n\rightarrow\infty}V(\phi_{M_n})=\alpha,$$
$$A\left(\tilde\phi\right)=\lim_{n\rightarrow\infty}A(\phi_{M_n})=\beta,$$
Thus $\tilde\phi\in S$. By convergence and semi-lower continuity, we have
$$E_{ADE}\left(\tilde\phi\right)\leq\lim_{n\rightarrow\infty}E_{ADE}(\phi_{M_n})\leq\lim_{n\rightarrow\infty}E_{M_n}(\phi_{M_n})\leq E_{ADE}(\phi^*)=\min_{\phi\in S}E_{ADE}(\phi).$$
So $\tilde\phi$ achieves the minimum of $E_{ADE}(\phi)$ with volume and surface area constraints. Moreover, the inequality becomes an equality, which shows that $\phi_{M_n}$ strongly converges to $\tilde\phi$.
\end{proof}

\vspace{3mm}
Therefore, we see that the original problem of minimizing the energy of the ADE model with prescribed surface area and bulk volume constraints can be formulated as finding the function $\phi=\phi(x)$ on the whole domain that minimizes the energy $E(\phi)$.

Let
$$T_1(\phi)=M_1\left(V(\phi)-\alpha\right)^2,\ \  T_2(\phi)=M_2\left(A(\phi)-\beta \right)^2,$$

The Allen-Cahn type dynamic equation takes the following form:
\begin{equation}\label{2.9}
\phi_t=-\frac{\delta E_M}{\delta\phi}=-\left(\frac{\delta W}{\delta\phi}+\frac{\delta G}{\delta\phi}+\frac{\delta T1}{\delta\phi}+\frac{\delta T2}{\delta\phi} \right).
\end{equation}
with the initial value $\phi(x,0)=\phi_0(x)$.
For the boundary conditions, we adopt periodic boundary conditions. While other approaches, such as a Dirichlet condition ($\phi = -1$) as used in \cite{du2004phase}, can effectively confine a vesicle, our choice is motivated by two key factors. Computationally, periodic boundary conditions are a natural requirement for the highly efficient Fourier-spectral method employed in our numerical schemes. Physically, this condition is well-justified for modeling a representative vesicle in a bulk environment, thereby avoiding artificial interactions with solid domain walls, especially when the vesicle occupies a significant portion of the domain.

Let us denote
\begin{equation*}
\left\{
\begin{array}{ll}
&f=\epsilon\Delta \phi-\frac{1}{\epsilon}(\phi^2-1)\phi,\\
&f_c=\epsilon\Delta\phi-\frac{1}{\epsilon}(\phi^2-1)(\phi+C\epsilon),\\
&g=\Delta f_c-\frac{1}{\epsilon^2}(3\phi^2+2C\epsilon \phi-1)f_c.
\end{array}
\right.
\end{equation*}
We provide the required variational derivatives:
\begin{align*}
&\frac{\delta W}{\delta \phi}=\kappa g,\\
&\frac{\delta G}{\delta\phi}=-\frac{3\overline{\kappa}\pi}{4A_0 D}( \Delta A-\Delta A_0)\left(-2\phi\Delta\phi-\Delta(\phi^2)+\frac{1}{\epsilon^2}(1-6\phi^2+5\phi^4)\right):=h,\\
&\frac{\delta T_1}{\delta\phi}=M_1\left(V(\phi)-\alpha\right),\\
&\frac{\delta T_2}{\delta\phi}=\frac{3\sqrt{2}}{2} M_2\left(A(\phi)-\beta\right)(-f).
\end{align*}

\section{Numerical schemes}
Next, we focus on the numerical solution of (\ref{2.9}). For the spatial discretization, we use Fourier spectral methods. Due to the regularization effect of the finite transition layer, for fixed $\epsilon$, the solutions exhibit high-order regularities. This property makes spectral methods, often implemented with FFT routines, very efficient for this problem. There are a number of options for the time discretization. One can use the forward Euler method:
\begin{equation}
\frac{\phi_{n+1}-\phi_n}{\Delta t}=-\left(\kappa g(\phi_n)+\frac{\delta G(\phi_n)}{\delta\phi}+\frac{\delta T1(\phi_n)}{\delta\phi}+\frac{\delta T2(\phi_n)}{\delta\phi}\right).
\end{equation}
The energy decay properties can be ensured, but only for a sufficiently small time step $\Delta t$. To improve the stability and accuracy of the temporal approximations while maintaining comparable efficiency, we can apply a semi-implicit time discretization scheme:
\begin{equation}\label{3.2}
\frac{\phi_{n+1}-\phi_n}{\Delta t}=-\left(\kappa g_{n,n+1}+\frac{\delta G(\phi_n)}{\delta\phi}+\frac{\delta T1(\phi_n)}{\delta\phi}+\frac{\delta T2(\phi_n)}{\delta\phi}\right),
\end{equation}
where
\begin{align*}
g_{n,n+1}=&\epsilon\Delta^2\phi_{n+1}+\frac{2}{\epsilon}\Delta \phi_{n+1}-\frac{1}{\epsilon}\Delta\phi_n^3-C\Delta\phi_n^2-\frac{3}{\epsilon}\phi_n^2\Delta\phi_n-2C\phi_n\Delta\phi_n\\
&+\frac{1}{\epsilon^3}(3\phi_n^2+2C\epsilon\phi_n-1)(\phi_n^2-1)(\phi_n+C\epsilon).
\end{align*}
In this scheme, the higher-order derivative terms in $\phi$ are treated implicitly, while the remaining nonlinear parts are treated explicitly.

Theoretically, we can also adopt a fully implicit scheme
\begin{align}\label{3.3}
\frac{\phi_{n+1}-\phi_n}{\Delta t}=&-\left\{   \kappa g(\phi_n,\phi_{n+1}) + h(\phi_n,\phi_{n+1})+\frac{M_1}{2} \left(V(\phi_{n+1})+V(\phi_n)-2\alpha \right) \right.  \\
& \left. + \frac{3\sqrt{2}M_2}{4} \left(A(\phi_{n+1})+A(\phi_n)-2\beta \right)\left(-f(\phi_n,\phi_{n+1})\right)   \right\}\notag
\end{align}
and ensure the monotonic decreasing of the energy while preserving the constraints. First, we define the function $f$, $g$ and $h$ as
\begin{align*}
&f(\phi,\eta)=\frac{\epsilon}{2}\Delta (\phi+\eta)-\frac{1}{4\epsilon}(\phi^2+\eta^2-2)(\phi+\eta),\\
&g(\phi,\eta)=\frac{1}{2}\Delta \left(f_c(\phi)+f_c(\eta)\right) - \frac{1}{2\epsilon^2}\left(\phi^2+\phi\eta+\eta^2 + C\epsilon(\phi+\eta)-1\right)  \left(  f_c(\phi)+f_c(\eta)  \right),\\
&h(\phi,\eta)=-\frac{3\overline{\kappa}\pi}{8A_0 D}\left( \Delta A(\phi_n)+\Delta A(\phi_{n+1})-2\Delta A_0  \right)\\
&\, \, \, \, \times
\left(-\phi\Delta\phi-\eta\Delta\eta-\frac{1}{2}\Delta(\phi^2+\eta^2)+\frac{1}{\epsilon^2}(1-2\phi^2-2\eta^2-2\phi\eta+2\phi^4+2\eta^4+\phi^2\eta^2)\right).
\end{align*}
These reformulated functions exhibit symmetry in both arguments.
This generalization of nonlinear terms in frameworks provides convenience for our analysis.

\vspace{3mm}
\begin{lemma}
The solution of (\ref{3.3}) satisfies
\begin{align*}
E_M(\phi_{n+1})-E_M(\phi_n)+\frac{1}{\Delta t}\int_{\Omega}(\phi_{n+1}-\phi_n)^2dx=0.
\end{align*}
\end{lemma}
\begin{proof}
For (\ref{3.3}), we have
\begin{equation}\label{3.4}
W(\phi_{n+1})-W(\phi_n)=\int_{\Omega}(\phi_{n+1}-\phi_n)\kappa g(\phi_{n+1},\phi_n)dx,
\end{equation}
\begin{equation}\label{3.5}
G(\phi_{n+1})-G(\phi_n)=\int_{\Omega}(\phi_{n+1}-\phi_n)h(\phi_{n+1},\phi_n)dx,
\end{equation}
\begin{align}\label{3.6}
&M_1\left(V(\phi_{n+1}\right)-\alpha)^2-M_1\left(V(\phi_{n})-\alpha\right)^2\\
=&M_1\left(V(\phi_{n+1})+V(\phi_n)-2\alpha\right)\left(V(\phi_{n+1})-V(\phi_n)\right)\notag\\
=&M_1\left(V(\phi_{n+1})+V(\phi_n)-2\alpha\right)\int_{\Omega}(\phi_{n+1}-\phi_n)\frac{1}{2}dx,\notag
\end{align}
\begin{align}\label{3.7}
&M_2\left(A(\phi_{n+1})-\beta\right)^2-M_2\left(A(\phi_{n})-\beta\right)^2\\
=&M_2\left(A(\phi_{n+1})+A(\phi_n)-2\beta\right)\left(A(\phi_{n+1})-A(\phi_n)\right)\notag\\
=&M_2\left(A(\phi_{n+1})+A(\phi_n)-2\beta\right)\int_{\Omega}(\phi_{n+1}-\phi_n)\left(-\frac{3\sqrt{2}}{4}f(\phi_{n+1},\phi_n)\right)dx.\notag
\end{align}

Putting together (\ref{3.4})-(\ref{3.7}), we have the discrete energy law for the numerical solution corresponding to (\ref{3.3}).
\end{proof}

We may also adopt a full backward Euler scheme
\begin{equation}\label{3.8}
\frac{\phi_{n+1}-\phi_n}{\Delta t}=-\left(\kappa g(\phi_{n+1})+\frac{\delta G(\phi_{n+1})}{\delta\phi}+\frac{\delta T1(\phi_{n+1})}{\delta\phi}+\frac{\delta T2(\phi_{n+1})}{\delta\phi}\right).
\end{equation}
As a result, the discrete energy law no longer holds in its strict form, and instead, we obtain the following.
\begin{proposition}
For all $\Delta t>0$ and a given $\phi^n\in H^2(\Omega)$, there exists a solution $\phi^{n+1}$ satisfying the backward Euler scheme (\ref{3.8}). Moreover, $\phi^{n+1}$ may be given by the minimizer in $H^2(\Omega)$ of the modified energy functional:
\begin{equation}\label{3.9}
E_M(\phi)+\frac{1}{2\Delta t}\int_{\Omega}(\phi-\phi^n)^2dx.
\end{equation}
Furthermore,
\begin{equation*}
E_M(\phi^{n+1})-E_M(\phi^n)+\frac{1}{2\Delta t}\int_{\Omega}(\phi^{n+1}-\phi^n)^2dx\leq0.
\end{equation*}
\end{proposition}
\noindent The proof is adapted from the similar result in Du and Wang \cite{du2007convergence}.

\section{Numerical experiments}
In this section, we present a series of numerical experiments to demonstrate the capabilities of the proposed model. All simulations are performed in a cubic domain $\Omega$ with periodic boundary conditions.
The spatial variables are discretized using a Fourier-spectral method with a cubic grid of grid size $h$. The temporal evolution is solved using the semi-implicit scheme described in (\ref{3.2}).
This scheme improves stability and allows for larger time steps, making it more efficient than fully explicit methods, especially for long-time simulations.
Unless otherwise specified, the model parameters are set as follows: $\Omega=[0,1]^3$, $h= \frac{1}{64}$, $\kappa=1$, $\bar{\kappa}=1.4$, $M_1=10^5$ and $M_2=10^4$.
For the numerical experiments, we only considered $C=0$ to showcase the influence of the ADE term only.
The physical distance between leaflets is defined as $D=\frac{2}{3}\epsilon$, where $\epsilon$ is the interface width. Each simulation is run until the system reaches a steady state.

\subsection{Formation of Classic Vesicle Shapes}
(1)Let $\epsilon=0.04$, $\Delta t=5\times10^{-7}$, $\alpha=0.0289$, $\beta=0.4880$, $\Delta A_0=0.1090$. The initial condition of the phase-field variable is given as
\begin{equation}
u_0=\tanh\left( \frac{0.35-(\frac{(x-0.5)^2}{0.5} + \frac{(y-0.5)^2}{0.5} + \frac{(z-0.5)^2}{0.1} )}{\sqrt{2}\epsilon} \right).
\end{equation}
The final steady-state configuration is a discocyte shape illustrated in Figure 1. Figure 1(a) shows the 3D view of the shape, while Figure 1(b) presents its corresponding cross-section. The cross-section clearly shows that the phase-field variable transitions smoothly between $-1$ and $1$, which validates the diffuse interface nature of our model. This behavior of the phase field is similar for all examples considered in this paper.
\begin{figure}[htbp]
\centering
\subfigure[3D view]{
\includegraphics[scale=0.4]{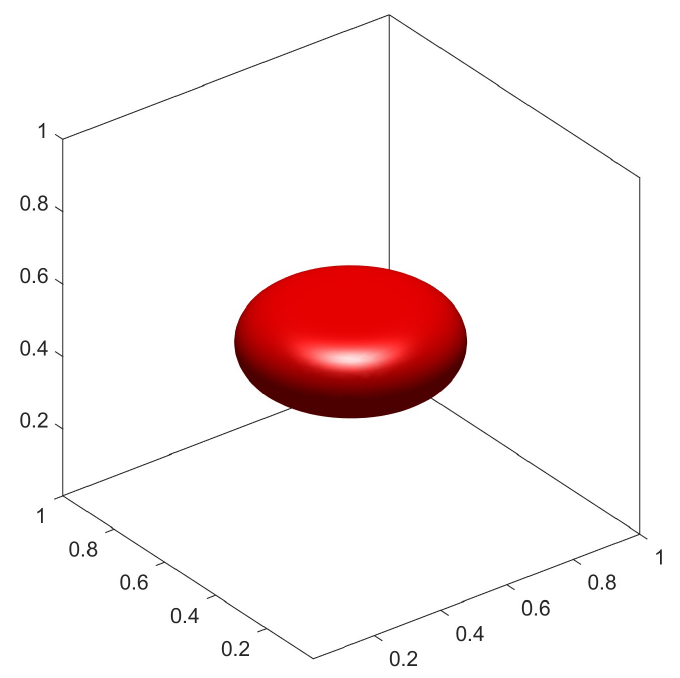}}%
\subfigure[corresponding cross section view]{
\includegraphics[scale=0.4]{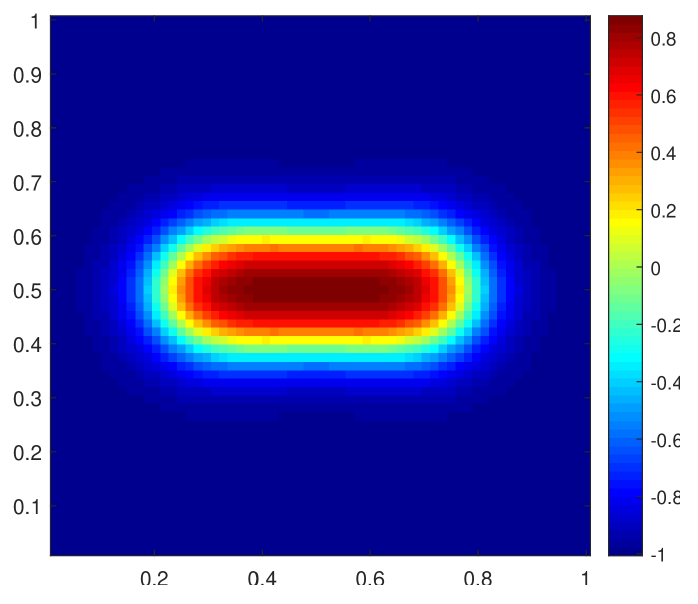}}
\caption{discocyte shape}
\end{figure}
\begin{remark}
The relaxed area difference, $\Delta A_0$, is assigned by first calculating the area difference $\Delta A$ of the initial configuration using equation (\ref{2.5}), and then systematically varying this calculated value for subsequent simulations.
\end{remark}

(2)Let $\epsilon=0.03$, $\Delta t=2\times10^{-7}$, $\alpha=0.0652$, $\beta=0.9092$, $\Delta A_0=0.2839$. The initial condition is
\begin{equation}\label{4.2}
u_0 = \tanh\left( \frac{0.6-(\frac{(x-0.5)^2}{0.35^2} + \frac{(y-0.5)^2}{0.35^2} + \frac{(z-0.5)^2}{0.35^2} )}{\sqrt{2}\epsilon} \right).
\end{equation}
The final steady-state configuration is a torus. Its 3D view and corresponding cross-section are illustrated in Figure 2(a) and 2(b), respectively.
\begin{figure}[htbp]
\centering
\subfigure[3D view]{
\includegraphics[scale=0.4]{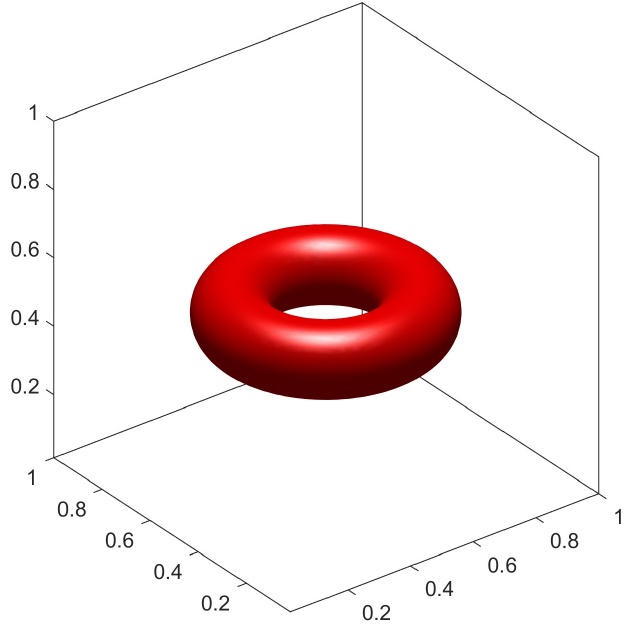}}%
\subfigure[corresponding cross section view]{
\includegraphics[scale=0.4]{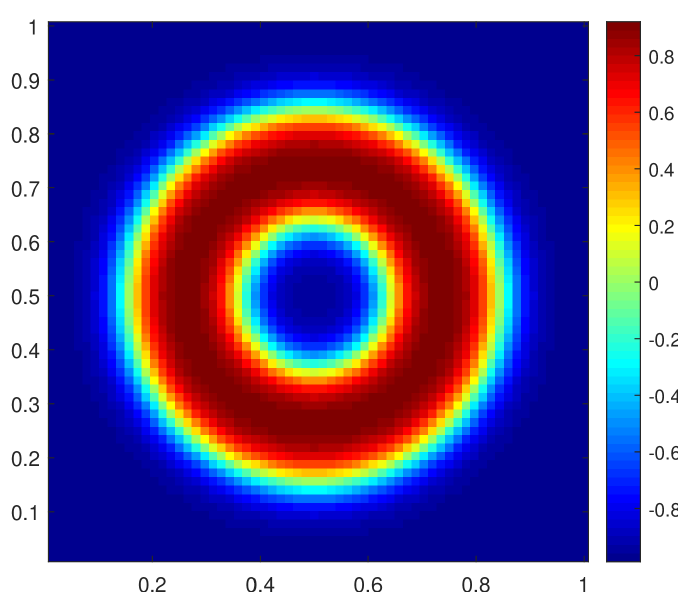}}
\caption{torus}
\end{figure}

Next, we investigate the continuous morphological transitions driven by systematically varying the surface area ($\beta$) and the relaxed area difference ($\Delta A_0$). In this series of experiments, the initial state is a flattened ellipsoid (\ref{4.3}) and the core parameters remain constant ($\epsilon=0.05$, $\alpha=0.0077$, $\Delta t=5\times10^{-7}$).
\begin{equation}\label{4.3}
u_0 = \tanh\left( \frac{0.5-(\frac{(x-0.5)^2}{0.2^2} + \frac{(y-0.5)^2}{0.2^2} + \frac{(z-0.5)^2}{0.35^2} )}{\sqrt{2}\epsilon} \right).
\end{equation}

(3)As we incrementally increase $\beta$ and $\Delta A_0$, the vesicle evolves from a simple biconcave shape (Figure 3(a)) into increasingly elongated gourd-like structures (Figures 3(b) and 3(c)). Further increasing $\Delta A_0$ while keeping $\beta$ constant causes the ``neck'' of the gourd to elongate until it forms a stable cylinder (Figure 3(e)). A final increase in $\Delta A_0$ results in a wider, more pronounced gourd shape (Figure 3(d)).The specific parameters for each shape are detailed in Table 1.
A similar transition was also observed experimentally in \cite{hollo2021shape}, where the authors tracked the reduction in volume, equivalent to the increase in surface area. Hence the results showcase the biological relevance of the formulation.
\begin{center}
{Table 1:\quad Parameters for Gourd and Cylinder Shape Transitions}\vskip 0.1cm
\begin{tabular}{c c c c}
\hline
$Figure$   &$\beta$  &$\Delta A_0$    &Resulting Shape   \\
\hline
$3(a)$     &0.1992   &0.1614       &Biconcave       \\
$3(b)$     &0.2068   &0.1676   &Early Gourd     \\
$3(c)$    &0.2390    &0.1906     &Elongated Gourd    \\
$3(d)$    &0.2390    &0.2253    &Gourd         \\
$3(e)$    &0.2390    &0.2426    &Cylinder      \\
\hline
\end{tabular}
\end{center}

\begin{figure}[htbp]
\centering
\subfigure[]
{
    \begin{minipage}[b]{.4\linewidth}
        \centering
        \includegraphics[scale=0.4]{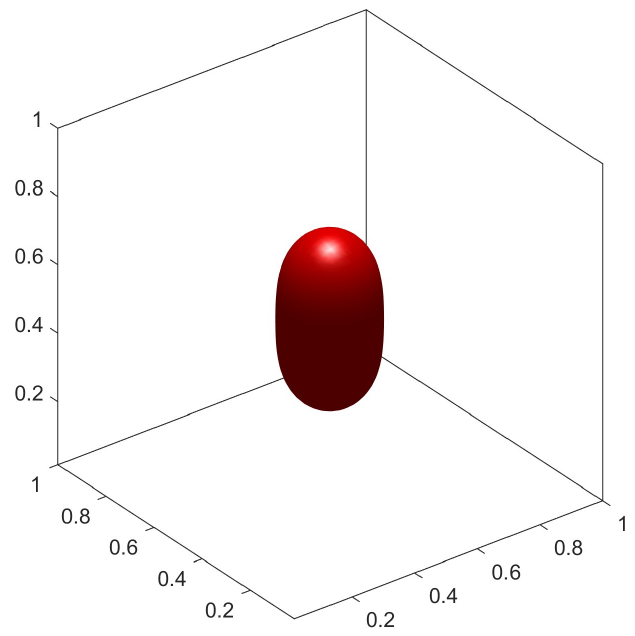}
    \end{minipage}
}
\subfigure[]
{ 	\begin{minipage}[b]{.4\linewidth}
        \centering
        \includegraphics[scale=0.4]{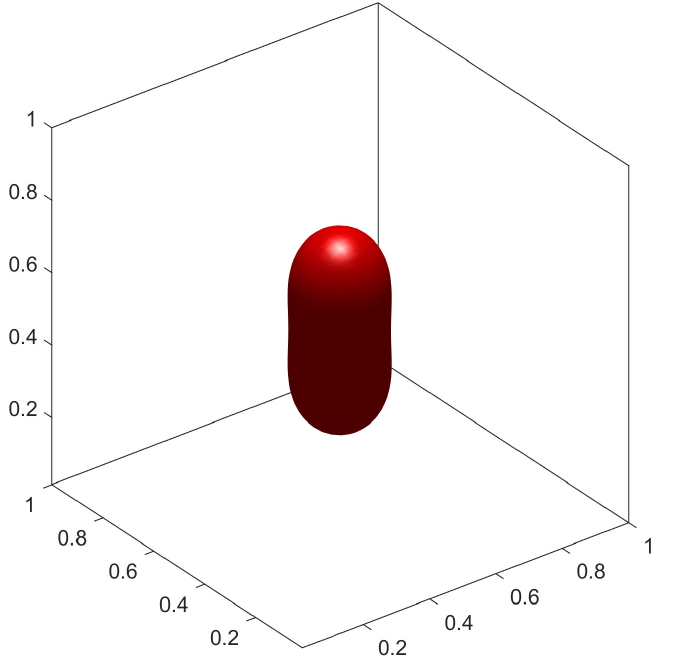}
    \end{minipage}
}
\subfigure[]
{
    \begin{minipage}[b]{.4\linewidth}
        \centering
        \includegraphics[scale=0.4]{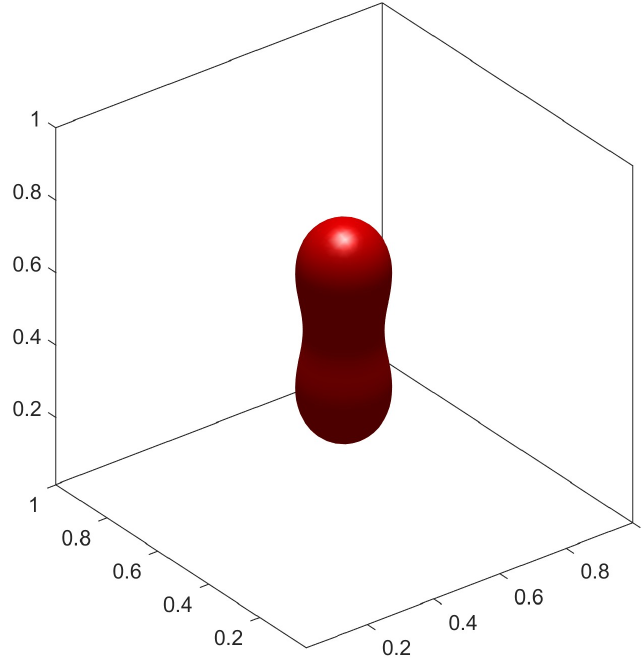}
    \end{minipage}
}
\subfigure[]
{
    \begin{minipage}[b]{.4\linewidth}
        \centering
        \includegraphics[scale=0.4]{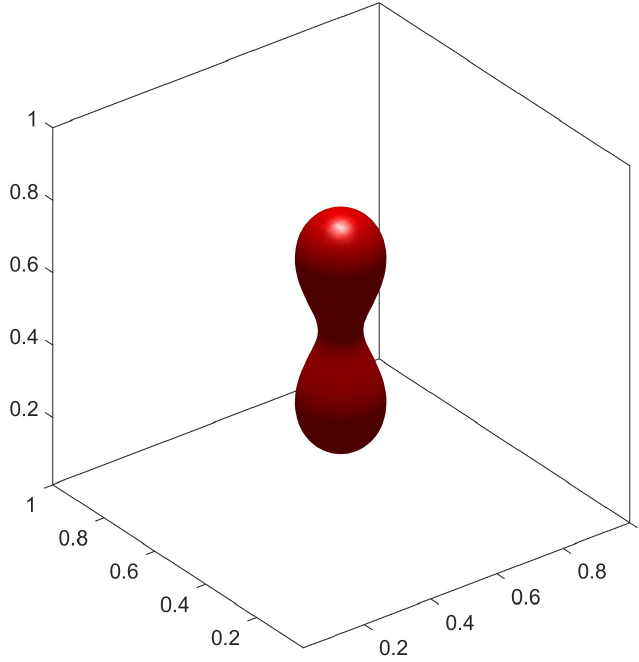}
    \end{minipage}
}
\subfigure[]
{
    \begin{minipage}[b]{.4\linewidth}
        \centering
        \includegraphics[scale=0.4]{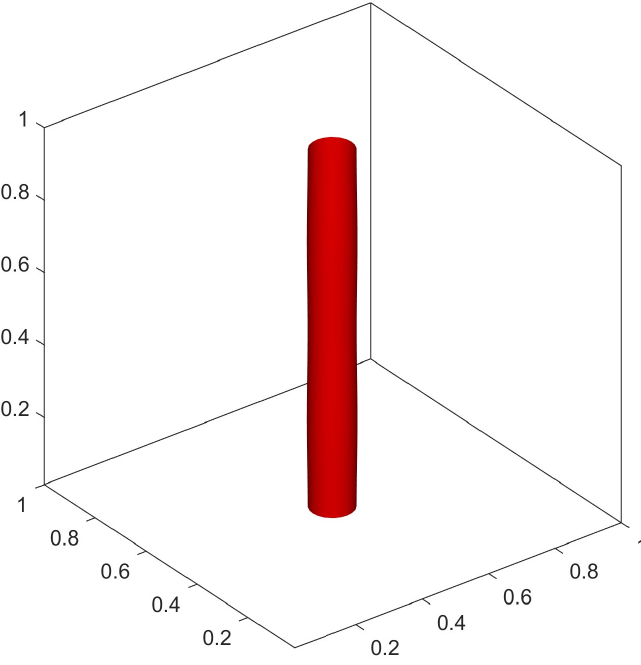}
    \end{minipage}
}
\caption{gourd shape}
\end{figure}

\subsection{Emergence of Complex, High-Genus Morphologies}
To showcase the model's versatility, we explored parameter regimes that give rise to more complex topologies. We found that these intricate, high-genus shapes are more readily formed with a smaller interface thickness ($\epsilon=0.02$). A smaller $\epsilon$ provides a better approximation of the sharp-interface bending energy, particularly in regions of high curvature that are characteristic of such complex structures. By further adjusting the volume and area constraints, we can stabilize a rich spectrum of configurations.
These include two connected spheres (Figure 4), resembling an early stage of cell division, and an elongated chain-like structure (Figure 5), which mimics vesicle fission.
Furthermore, by adjusting the initial geometry and constraints, we can generate multi-armed, starfish-like configurations. Figures 6, 7, and 8 show stable steady states with three, four, and six arms, respectively. The formation of these complex, high-genus structures highlights the model's capability to capture sophisticated membrane behaviors beyond simple deformations. This makes it a powerful tool for exploring phenomena such as the formation of tubular networks in cellular organelles\cite{voeltz2002structural,shibata2006rough}.
 The specific parameters are detailed in experiments (4) to (9).

(4)Let $\epsilon=0.02$, $\Delta t=2\times10^{-7}$, $\alpha=0.0074$, $\beta=0.1969$, $\Delta A_0=0.0711$. Other parameters follow the same configuration as (3).
The final steady-state shape is two spheres connected in Figure 4.
\begin{figure}[htbp]
\centering
\includegraphics[scale=0.5]{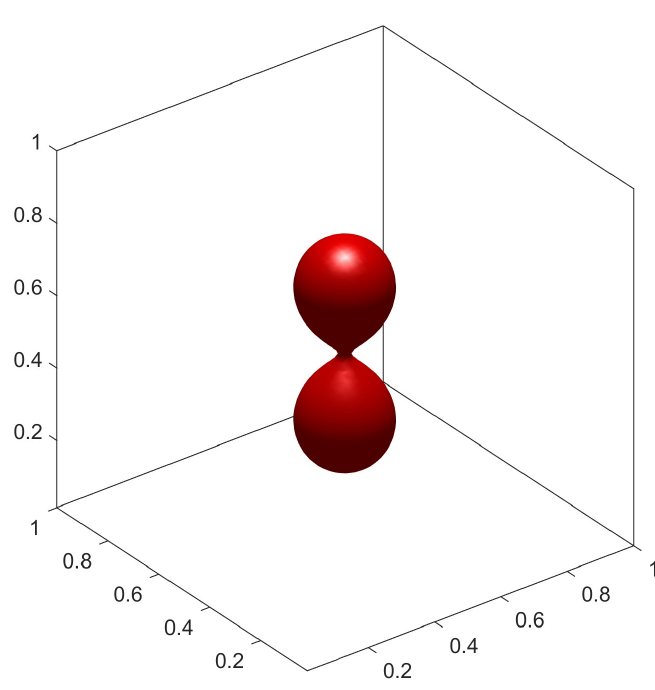}
\caption{two-sphere connected shape}
\end{figure}

(5)Let $\epsilon=0.02$, $\Delta t=5\times10^{-7}$, $\alpha=0.0074$, $\beta=0.2328$, $\Delta A_0=0.0958$. The initial condition is
\begin{equation}
u_0= \tanh\left( \frac{0.5-(\frac{(x-0.5)^2}{0.2^2} + \frac{(y-0.5)^2}{0.2^2} + \frac{(z-0.5)^2}{0.35^2} )}{\sqrt{2}\epsilon} \right).
\end{equation}
The final steady-state configuration is a chain-shaped illustrated in Figure 5.
\begin{figure}[htbp]
\centering
\includegraphics[scale=0.4]{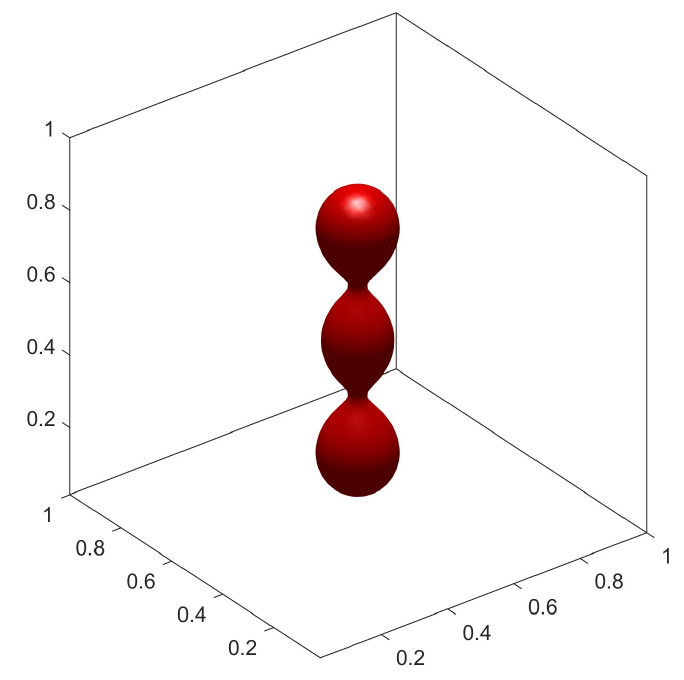}
\caption{chain-shaped}
\end{figure}

(6)Let $\epsilon=0.02$, $\Delta t=1\times10^{-7}$, $\alpha=0.0226$, $\beta=0.4489$, $\Delta A_0=0.1520$. The initial condition is
\begin{equation}
u_0= \tanh\left( \frac{0.5-(\frac{(x-0.5)^2}{0.35^2} + \frac{(y-0.5)^2}{0.35^2} + \frac{(z-0.5)^2}{0.35^2} )}{\sqrt{2}\epsilon} \right).
\end{equation}
The final steady-state configuration is a three-armed shape. The result is illustrated in Figure 6.
\begin{figure}[htbp]
\centering
\includegraphics[scale=0.4]{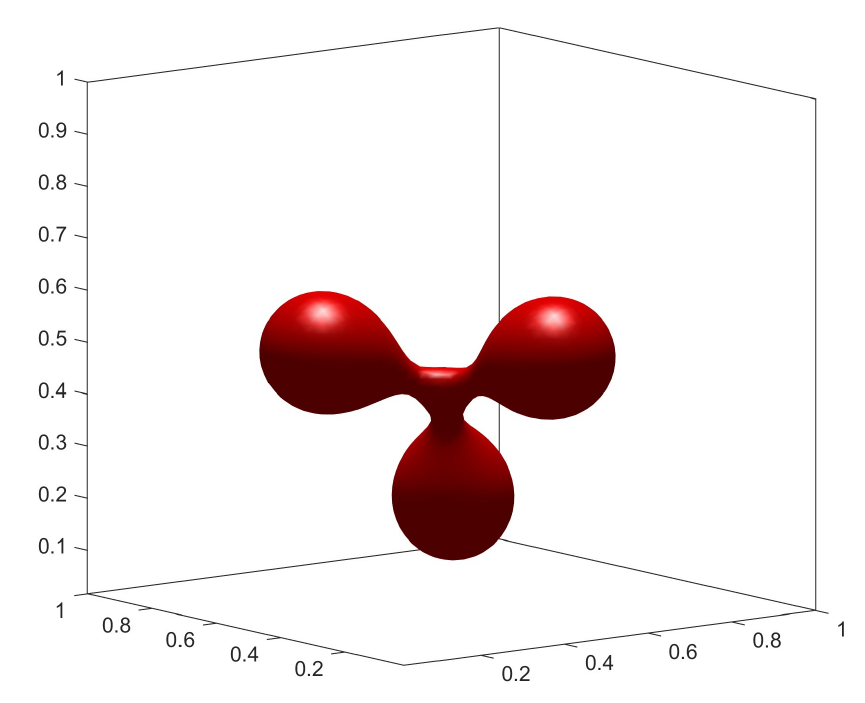}
\caption{three-armed shape}
\end{figure}

(7)Let $\epsilon=0.02$, $\Delta t=2\times10^{-7}$, $\alpha=0.0097$, $\beta=0.2550$, $\Delta A_0=0.1146$. The initial condition is
\begin{equation}
u_0= \tanh\left( \frac{0.5-(\frac{(x-0.5)^2}{0.35^2} + \frac{(y-0.5)^2}{0.35^2} + \frac{(z-0.5)^2}{0.15^2} )}{\sqrt{2}\epsilon} \right).
\end{equation}
The final steady-state configuration is a four-armed shape in Figure 7.
\begin{figure}[htbp]
\centering
\includegraphics[scale=0.4]{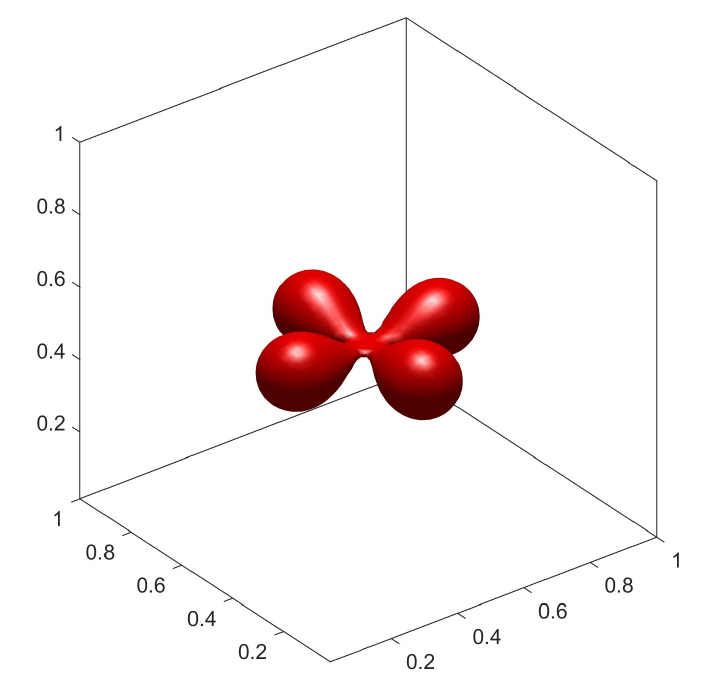}
\caption{four-armed shape}
\end{figure}

(8)Let $\epsilon=0.02$, $\Delta t=1\times10^{-7}$, $\alpha=0.0529$, $\beta=0.7911$, $\Delta A_0=0.1766$. The initial condition is same as in (\ref{4.2}).
The final steady-state configuration is a six-armed shape illustrated in Figure 8.
\begin{figure}[htbp]
\centering
\includegraphics[scale=0.4]{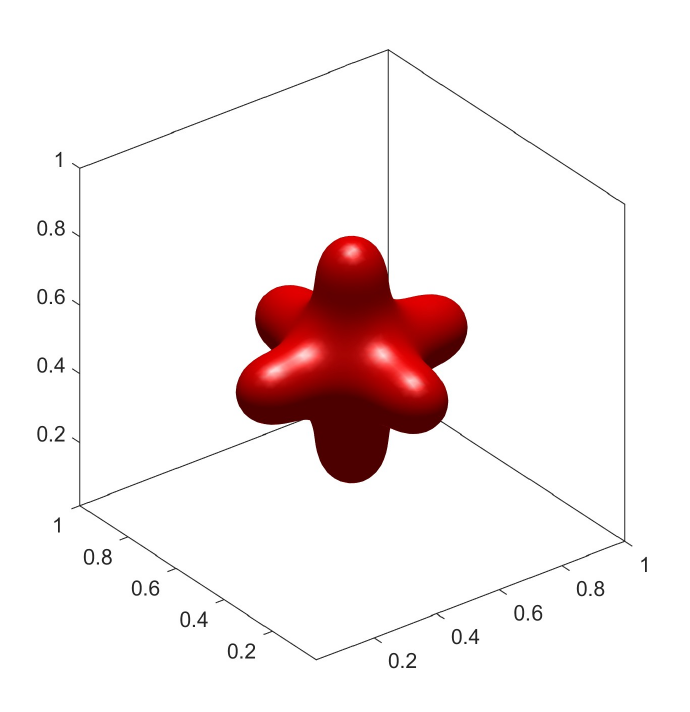}%
\caption{six-armed shape}
\end{figure}

We also explore the formation of nested shapes \cite{salva2013polymersome} within vesicle membranes. By adjusting initial condition, vesicles can form intricate configurations such as multi-chambered structures or internal spherical shapes. As shown in the following example:

(9)Let $\Omega=[0,2]^3$, $h=\frac{1}{50}$, $\epsilon=0.03$, $M_1=M_2=10^4$, $\Delta t=5\times10^{-7}$, $\alpha=0.2693$, $\beta=2.8347$, $\Delta A_0=0.3939$. The initial condition is
\begin{equation}
u_0= \tanh\left( \frac{1-(\frac{(x-1)^2}{0.16} + \frac{(y-1)^2}{0.16} + \frac{(z-1)^2}{0.16} )}{\sqrt{2}\epsilon} \right).
\end{equation}
The final steady-state configuration emerges where two spheres are embedded within one another, forming a nested vesicle(Figure 9). This complex topology showcases how the interplay between membrane constraints can lead to novel geometries, which may be relevant for modeling vesicle fission or cell division processes.
\begin{figure}[htbp]
\centering
\includegraphics[scale=0.4]{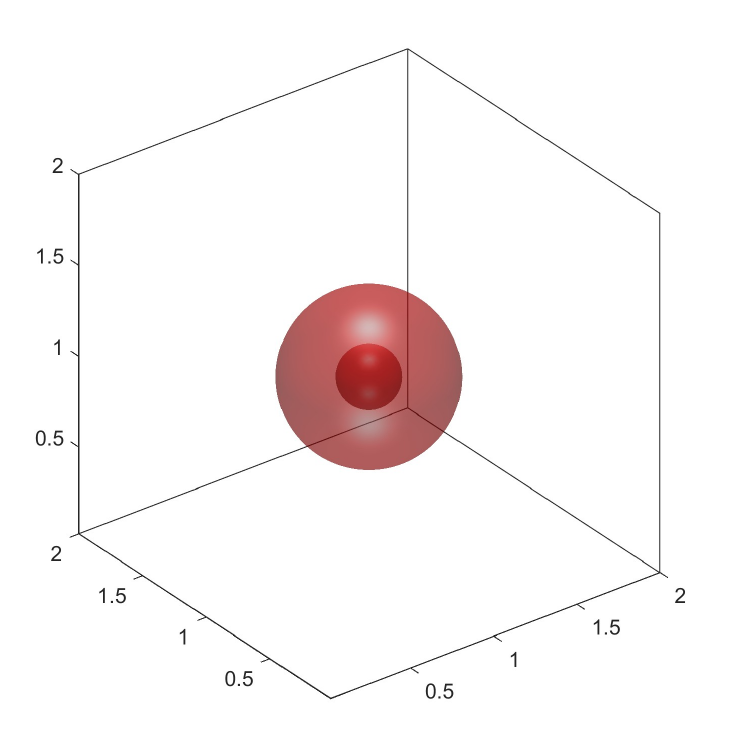}%
\caption{nested shape}
\end{figure}

\section{Conclusion}
In this paper, we developed and numerically implemented a phase-field model for vesicle membranes that incorporates area-difference elasticity while enforcing constraints on volume and surface area.
Our model, solved by an efficient spectral method, successfully predicts a wide range of equilibrium morphologies.
The numerical results demonstrate a clear pathway of shape transformations from simple discocytes to complex structures like gourd shape, cylinders, multi-armed and nested configurations, primarily governed by the competition between bending energy and the ADE constraint.
Future work will focus on further analysis of the capabilities of the proposed model by adopting biologically relevant dimensions of vesicles. As mentioned in \cite{salva2013polymersome}, the sharp interface models fail to capture some experimentally observed steady-state shapes of the vesicles when the vesicle radius is below a threshold value around four times the membrane thickness. One such shape is the nested vesicle, which our formulation was able to capture effectively. We anticipate this to be due to the diffuse interface thickness having a biological relevance, which will be further explored. In addition, the proposed model will be further developed to include the phase field formulation of membrane cytoskeleton properties, essential for capturing the steady-state shapes of red blood cells \cite{lim2002stomatocyte}.

\section{Acknowledgements}
This work was supported by the National Natural Science Foundation of China 12371388.


\bibliographystyle{unsrt}  
\bibliography {reference}

\end{document}